\documentclass[12pt]{amsart}

\usepackage{amssymb, amsmath, amsthm, amscd, mathrsfs, graphicx}
\usepackage[cmtip, all] {xy}
\usepackage{placeins}
\usepackage{eufrak}
\numberwithin{equation}{section}
\newtheorem{teo}{Theorem}[section]
\newtheorem{pro}[teo]{Proposition}

\newtheorem{cor}[teo]{Corollary}

\theoremstyle{definition}

\newtheorem{definition}[teo]{Definition}
\newtheorem{exa}[teo]{Example}

\newtheorem{remark}[teo]{Remark}

\theoremstyle{remark}

\def\&{\wedge}
\newcommand{\w}{\omega}
\newcommand{\R}{\mathbb{R}}
\newcommand{\e}{\mathbf{e}}

\newcommand{\bv}{\mathbf{v}}
\newcommand{\bw}{\mathbf{w}}
\newcommand{\bx}{\mathbf{x}}
\newcommand{\bb}{\mathbb}

\begin{document}

\title{The geometry of lightlike surfaces in Minkowski space}

\author{Brian Carlsen}
\address{Department of Mathematics, 395 UCB,U niversity of
Colorado, Boulder, CO 80309-0395}
\email{Brian.Carlsen@colorado.edu}

\author{Jeanne N. Clelland}
\address{Department of Mathematics, 395 UCB, University of
Colorado, Boulder, CO 80309-0395}
\email{Jeanne.Clelland@colorado.edu}

\subjclass[2010]{Primary(53A35), Secondary(51B20)}
\keywords{lightlike surfaces, Minkowski space, method of moving frames}
\thanks{This research was supported in part by NSF grants DMS-0908456 and DMS-1206272.}

\begin{abstract}
We investigate the geometric properties of lightlike surfaces in the Minkowski space $\R^{2,1}$, using Cartan's method of moving frames to compute a complete set of local invariants for such surfaces.  Using these invariants, we give a complete local classification of lightlike surfaces of constant type in $\R^{2,1}$ and construct new examples of such surfaces.  
\end{abstract} 

\maketitle

\section{Introduction}

Lightlike surfaces in Minkowski space have been the subject of much attention in recent years, largely due to their importance in the theory of general relativity.  Some classification results have been obtained for special categories of lightlike surfaces; for example, in \cite{BFL98} it is shown that an entire lightlike hypersurface with zero lightlike mean curvature must be a lightlike hyperplane, while in \cite{IL09} it is shown that the only lightlike surfaces of revolution are lightlike planes and light cones.

In this paper we undertake a much more general classification problem: our main result is the complete local classification of lightlike surfaces of constant type in $\R^{2,1}$ (cf. Theorem \ref{main-theorem}).  (The precise definition of ``constant type" will become clear during the classification process; it is roughly similar to the notion of a regular surface in $\R^3$ containing no umbilic points.)  We use Cartan's method of equivalence to construct a complete set of local invariants for lightlike surfaces of constant type; we then use these invariants to construct canonical local coordinates for lightlike surfaces and give a local normal form for lightlike surfaces parametrized with respect to these coordinates.

The paper is organized as follows: in \S \ref{background-sec} we briefly review the necessary definitions and introduce the notion of adapted frames for lightlike surfaces, together with the associated Maurer-Cartan forms; in \S \ref{equivalence-prob-sec} we carry out the equivalence method to compute local invariants for lightlike surfaces of constant type; in \S \ref{normal-forms-sec} we derive a local normal form for parametrized lightlike surfaces, and in \S \ref{examples-sec} we use this normal form to construct new examples of lightlike surfaces.  Some of the key features of our classification are:
\begin{itemize}
\item In addition to lightlike planes and light cones, we find that there is a much larger family of lightlike surfaces which we call {\em non-conical}; in fact, the family of non-conical lightlike surfaces is locally parametrized by one arbitrary function $f(v)$ of a single variable.
\item The non-conical lightlike surface associated to a given function $f(v)$ is closely related to the Sturm-Liouville operator $\left(\frac{d}{dt}\right)^2 - \tfrac{1}{2} f(v)$; specifically, the coordinate functions of the parametrization for the associated surface are given in terms of solutions to the Sturm-Liouville equation $h''(v) - \tfrac{1}{2} f(v) h(v) = 0$.
\item Every lightlike surface of constant type in $\R^{2,1}$ is ruled by null lines.
\end{itemize}

\section{Lightlike surfaces in $\R^{2,1}$, adapted frames, and Maurer-Cartan forms}\label{background-sec}

\begin{definition}\label{R21-def}
Three-dimensional {\em Minkowski space} is the manifold $\R^{2,1}$ defined by
\[ \R^{2,1} = \left\{ \bx = \begin{bmatrix} x^0 \\[0.05in] x^1 \\[0.05in] x^2 \end{bmatrix} : x^0, x^1, x^2 \in \bb{R} \right\}, \]
with the indefinite inner product $\langle \cdot, \cdot \rangle$ defined on each tangent space $T_{\bx}\R^{2,1}$ by
\[ \langle \bv, \bw \rangle = -v^0 w^0 + v^1 w^1 + v^2 w^2. \]
\end{definition}

\begin{definition}
A nonzero vector $\bv \in T_{\bx}\bb{R}^{2,1}$ is called: 
\begin{itemize}
\item {\em spacelike} if $\langle \bv, \bv \rangle > 0$;
\item {\em timelike} if $\langle \bv, \bv \rangle < 0$;
\item {\em lightlike} or {\em null} if $\langle \bv, \bv \rangle = 0$.
\end{itemize}
In particular, the set of all null vectors $\bv \in T_{\bx}\bb{R}^{2,1}$ forms a cone in the tangent space $T_{\bx}\bb{R}^{2,1}$, called the {\em light cone} or {\em null cone} at $\bx$.
\end{definition}

\begin{definition}
A regular surface $\Sigma \subset \bb{R}^{2,1}$ is called:
\begin{itemize}
\item {\em spacelike} if the restriction of $\langle \cdot, \cdot \rangle$ to each tangent plane $T_{\bx}\Sigma$ is positive definite;
\item {\em timelike} if the restriction of $\langle \cdot, \cdot \rangle$ to each tangent plane $T_{\bx}\Sigma$ is indefinite;
\item {\em lightlike} if the restriction of $\langle \cdot, \cdot \rangle$ to each tangent plane $T_{\bx}\Sigma$ is degenerate.
\end{itemize}
\end{definition}

We will use Cartan's method of moving frames to compute local invariants for lightlike surfaces $\Sigma \subset \R^{2,1}$ under the action of the Minkowksi isometry group, which consists of all transformations $\varphi: \R^{2,1} \to \R^{2,1}$ of the form
\begin{equation}
 \varphi(\bx) = A\bx + \mathbf{b}, \label{Minkowski-isometry}
\end{equation}
where $A \in O(2,1)$ and $\mathbf{b} \in \R^{2,1}$.  
%(See \cite{Gardner} for details about the method of equivalence.)  
Ordinarily one might begin by considering the set of orthonormal frames $(\e_0, \e_1, \e_2)$ for the tangent space $T_{\bx} \R^{2,1}$ at each point $\bx \in \Sigma$. However, because the restriction of the Minkowski metric $\langle \cdot, \cdot \rangle$ to each tangent space $T_{\bx}\Sigma$ of a lightlike surface is degenerate, orthonormal frames are not easily adaptable to the geometry of lightlike surfaces.  Instead, we will choose frames $(\e_0, \e_1, \e_2)$ based at each point $\bx \in \Sigma$ as follows: 
\begin{enumerate}
\item Since each tangent plane $T_{\bx}\Sigma$ is lightlike, it must contain a unique null direction.  Choose $\e_0$ to be any nonzero vector parallel to this direction.
\item Every nonzero vector in $T_{\bx}\Sigma$ which is linearly independent from $\e_0$ is necessarily spacelike; moreover, $T_{\bx}\Sigma$ consists precisely of all vectors in $T_{\bx}\R^{2,1}$ which are orthogonal to $\e_0$.  Choose $\e_1$ to be any nonzero vector in $T_{\bx}\Sigma$ of unit length; i.e., $\langle \e_1, \e_1 \rangle = 1$.
\item Since $\e_1$ is spacelike, the set of all vectors in $T_{\bx}\R^{2,1}$ which are orthogonal to $\e_1$ is a timelike plane and therefore contains two linearly independent null directions.  One of these directions is $\e_0$; choose $\e_2$ to be a nonzero vector parallel to the other null direction in this plane.  
\item Since $\e_2 \notin T_{\bx}\Sigma$, we must have $\langle \e_0, \e_2 \rangle \neq 0$.  By scaling $\e_2$, we can arrange that $\langle \e_0, \e_2 \rangle = 1$.
\end{enumerate}
Taken together, the conditions above say that the vectors $(\e_0, \e_1)$ span the tangent plane $T_{\bx}\Sigma$ at each point $\bx \in \Sigma$, and the frame vectors $(\e_0, \e_1, \e_2)$ satisfy the inner product relations:
\begin{gather}
 \langle \textbf{e}_0, \textbf{e}_0 \rangle = \langle \textbf{e}_0, \textbf{e}_1 \rangle =   \langle \textbf{e}_1, \textbf{e}_2 \rangle = \langle \textbf{e}_2, \textbf{e}_2 \rangle = 0 ,\label{adapted-frame-inner-products}
\\
\langle \textbf{e}_0, \textbf{e}_2 \rangle = \langle \textbf{e}_1, \textbf{e}_1 \rangle = 1. \notag
\end{gather}
We will call a frame satisfying these conditions \emph{0-adapted}.

Now suppose that $(\e_0, \e_1, \e_2), (\tilde{\e}_0, \tilde{\e}_1, \tilde{\e}_2)$ are two 0-adapted frames on a lightlike surface $\Sigma$.  According to the conditions above: 
\begin{itemize}
\item $\tilde{\e}_0$ must be a nonzero multiple of $\e_0$; say, $\tilde{\e}_0 = \mu \e_0$ for some nonvanishing function $\mu$ on $\Sigma$.  
\item From the inner product conditions \eqref{adapted-frame-inner-products} and the fact that $\tilde{\e}_1 \in T_{\bx}\Sigma$, we must have
\[ \tilde{\e}_1 = \pm \e_1 + \lambda \e_0 \]
for some real-valued function $\lambda$ on $\Sigma$. The sign ambiguity can be avoided by requiring both frames to be positively oriented, so without loss of generality, we can assume that
\[ \tilde{\e}_1 = \e_1 + \lambda \e_0. \]
\item It then follows from the inner product conditions \eqref{adapted-frame-inner-products} that
\[ \tilde{\e}_2 = \frac{1}{\mu} \e_2 - \frac{\lambda}{\mu} \e_1 -\frac{\lambda^2}{2\mu} \e_0. \]
\end{itemize}
Thus we see that any two (oriented) 0-adapted frames
must differ by a transformation of the form
\begin{equation}
\begin{bmatrix} \tilde{\e}_0 & \tilde{\e}_1 & \tilde{\e}_2 \end{bmatrix} = 
\begin{bmatrix} \e_0 & \e_1 & \e_2 \end{bmatrix}
\begin{bmatrix} \mu & \lambda & -\frac{\lambda^2}{2\mu} \\[0.05in]
0 & 1 & -\frac{\lambda}{\mu} \\[0.05in]
0 & 0 & \frac{1}{\mu} \end{bmatrix}.  \label{0-adapted-frame-transformation}
\end{equation}
Furthermore, if $(\e_0, \e_1, \e_2)$ is any 0-adapted frame on $\Sigma$, then any frame $(\tilde{\e}_0, \tilde{\e}_1, \tilde{\e}_2)$ given by \eqref{0-adapted-frame-transformation} is also 0-adapted.

\begin{remark}
The 0-adapted frames on $\Sigma$ may be regarded as the local sections of a principal fiber bundle $\pi: \mathcal{B}_0 \to \Sigma$, with structure group $G_0 \subset GL(3)$ consisting of all invertible matrices of the form in \eqref{0-adapted-frame-transformation}.  The group $G_0$ is referred to as the the {\em structure group} for the 0-adapted frames on $\Sigma$.
\end{remark}

The {\em Maurer-Cartan forms} $\w^{\alpha}, \w^{\alpha}_{\beta}$ associated to a 0-adapted frame $(\e_0, \e_1, \e_2)$ on $\Sigma$ are the 1-forms on $\Sigma$ defined by the equations:
\begin{align}
d\bx & = \e_{\alpha} \w^{\alpha}, \label{MC-forms-def}   \\ 
d\e_{\beta} & =  \e_{\alpha} \w^{\alpha}_{\beta}. \notag
\end{align}
(Note that all indices range from 0 to 2, and we use the Einstein summation convention.)  The 1-forms $\w^0, \w^1, \w^2$ are called the {\em dual forms} (or sometimes the {\em solder forms}), while the 1-forms $\{\w^{\alpha}_{\beta}, \ 0 \leq \alpha, \beta \leq 2\}$ are called the {\em connection forms}.  They satisfy the Maurer-Cartan structure equations:
\begin{align}
d\w^{\alpha} & = - \w^{\alpha}_{\beta} \& \w^{\beta},  \label{structure-eqns} \\
d\w^{\alpha}_{\beta} & = - \w^{\alpha}_{\gamma} \& \w^{\gamma}_{\beta}. \notag
\end{align}
(See \cite{Ivey} for a discussion of Maurer-Cartan forms and their structure equations.)  Differentiating the inner product relations \eqref{MC-forms-def} yields the following relations among the connection forms:
\begin{gather}
\w^2_0 = \w^1_1 = \w^0_2 = 0, \label{MC-forms-relations} 
\\
 \w^2_1 = -\w^1_0 , \qquad \w^2_2 = -\w^0_0, \qquad \w^1_2 = -\w^0_1. \notag
\end{gather}
Thus we can write the matrix $\Omega = [\w^{\alpha}_{\beta}]$ of connection forms as
\[\Omega = \begin{bmatrix} 
			\omega_0^0 & \omega_1^0 & \omega_2^0 \\[0.05in]
			\omega_0^1 & \omega_1^1 & \omega_2^1 \\[0.05in]
			\omega_0^2 & \omega_1^2 & \omega_2^2 
		\end{bmatrix}
		=
		\begin{bmatrix}
			\omega_0^0 & \omega_1^0 & 0 \\[0.05in]
			\omega_0^1 & 0 & -\omega_1^0 \\[0.05in]
			0  & - \omega_0^1 & - \omega_0^0 
		\end{bmatrix} .
\]
If $(\e_0, \e_1, \e_2), (\tilde{\e}_0, \tilde{\e}_1, \tilde{\e}_2)$ are two 0-adapted frames related by the equation
\[ \begin{bmatrix} \tilde{\e}_0 & \tilde{\e}_1 & \tilde{\e}_2 \end{bmatrix} = 
\begin{bmatrix} \e_0 & \e_1 & \e_2 \end{bmatrix} A \]
(with $A$ as in \eqref{0-adapted-frame-transformation}), with associated Maurer-Cartan forms $(\w^{\alpha}, \w^{\alpha}_{\beta}), 
(\tilde{\w}^{\alpha}, \tilde{\w}^{\alpha}_{\beta})$, respectively, then  equations \eqref{MC-forms-def} imply that
\begin{gather*}
 \begin{bmatrix} \tilde{\w}^0 \\[0.05in] \tilde{\w}^1 \\[0.05in] \tilde{\w}^2 \end{bmatrix} = A^{-1} \begin{bmatrix} \w^0 \\[0.05in] \w^1 \\[0.05in] \w^2 \end{bmatrix},  \\[0.1in]
 \begin{bmatrix} \tilde{\omega}_0^0 & \tilde{\omega}_1^0 & 0 \\[0.05in]\tilde{\omega}_0^1 & 0 & -\tilde{\omega}_1^0 \\[0.05in] 0  & - \tilde{\omega}_0^1 & - \tilde{\omega}_0^0 \end{bmatrix} = A^{-1} dA + 
A^{-1}  \begin{bmatrix} \omega_0^0 & \omega_1^0 & 0 \\[0.05in]\omega_0^1 & 0 & -\omega_1^0 \\[0.05in] 0  & - \omega_0^1 & - \omega_0^0 \end{bmatrix} A. 
\end{gather*}
More concretely, we have:
\begin{align}
\tilde{\w}^0 & = \frac{1}{\mu} \w^0 - \frac{\lambda}{\mu} \w^1 - \frac{\lambda^2}{2\mu} \w^2, \notag \\
\tilde{\w}^1 & = \w^1 + \lambda \w^2, \notag \\
\tilde{\w}^2 & = \mu \w^2, \label{0-adapted-MC-forms-transformation} \\
\tilde{\w}^0_0 & = \w^0_0 - \lambda \w^1_0 + \frac{1}{\mu} d\mu, \notag \\
\tilde{\w}^1_0 & = \mu \w^1_0 \notag ,\\
\tilde{\w}^0_1 & = \frac{1}{\mu} \left(\w^0_1 + \lambda \w^0_0 - \frac{1}{2} \lambda^2 \w^1_0 + d\lambda   \right) . \notag
\end{align}

\section{Reduction of the structure group and local invariants}\label{equivalence-prob-sec}

The method of moving frames proceeds by considering relations among the Maurer-Cartan forms on $\Sigma$.  From the equation
\[ d\bx = \e_0 \w^0 + \e_1 \w^1 + \e_2 \w^2 \]
and the fact that $d\bx$ takes values in the tangent space $T_{\bx}\Sigma$ at each point $\bx \in \Sigma$, it follows that $\w^2 = 0$, and that $\w^0, \w^1$ are linearly independent 1-forms which span the cotangent space $T^*_{\bx}\Sigma$ at each point $\bx \in \Sigma$. Differentiating the equation $\w^2=0$ and applying the relations \eqref{MC-forms-relations} yields
\begin{align*}
 0 = d\w^2 & = -\w^2_0 \& \w^0 - \w^2_1 \& \w^1 \\
 & = \w^1_0 \& \w^1.
\end{align*}
By Cartan's lemma (see \cite{Ivey}), we have
\[ \omega^1_0 = a_1 \omega^1 \]
for some function $a_1$ on $\Sigma$.  

Now suppose that we perform a transformation of the form \eqref{0-adapted-frame-transformation}.  According to \eqref{0-adapted-MC-forms-transformation}, the Maurer-Cartan forms $(\tilde{\w}^{\alpha}, \tilde{\w}^{\alpha}_{\beta})$ associated to the new frame satisfy
\begin{align*}
\tilde{\w}^1_0 & = \mu \w^1_0 \\
& = \mu a_1 \w^1 \\
& = \mu a_1 \tilde{\w}^1.
\end{align*}
Therefore, the transformed function $\tilde{a}_1$, defined by the equation
\[ \tilde{\omega}_1^0 = \tilde{a}_1 \tilde{\omega}^1, \]
is given by
\begin{equation}
 \tilde{a}_1 = \mu a_1. \label{a1-transformation}
\end{equation}

According to \eqref{a1-transformation}, the function $a_1$ is a {\em relative invariant} of the surface $\Sigma$: at each point $\bx \in \Sigma$, $a_1$ is either equal to zero for every 0-adapted frame based at $\bx$ or nonzero for every 0-adapted frame based at $\bx$.  In order to proceed with the method of moving frames, we must assume that $\Sigma$ is of {\em constant type} with respect to $a_1$; i.e., that $a_1$ is either identically zero on $\Sigma$ or never equal to zero at any point of $\Sigma$.  The latter assumption is similar in flavor to assuming that a regular surface in $\R^3$ is free of umbilic points; in practice, it simply means that we must restrict our attention to the open subset of $\Sigma$ where $a_1$ is nonzero.

First we consider the case where $a_1 \equiv 0$.

\begin{pro}\label{a1-zero-prop}
Let $\Sigma$ be a connected, regular lightlike surface in $\R^{2,1}$, and suppose that $a_1 \equiv 0$ for every 0-adapted frame on $\Sigma$.  Then $\Sigma$ is contained in a lightlike plane in $\R^{2,1}$.
\end{pro}

\begin{proof}
If $a_1 \equiv 0$ on $\Sigma$, then we have $\w^1_0 = 0$, and therefore
\[ d\e_0 = \e_0 \w^0_0. \]
In particular, the line spanned by $\e_0$ is constant on $\Sigma$.
Choose a point $\bx_0 \in \Sigma$, and consider the function
\[ f(\bx) = \langle \bx - \bx_0, \e_0 \rangle \]
on $\Sigma$.  We have
\begin{align*}
df & = \langle d\bx, \e_0 \rangle + \langle \bx - \bx_0, d\e_0 \rangle \\
& = \langle \e_0 \w^0 + \e_1 \w^1, \e_0 \rangle + \langle \bx - \bx_0, \e_0 \w^0_0 \rangle \\
& = f \w^0_0.
\end{align*}
Since $f(\bx_0) = 0$, the existence/uniqueness theorem for ordinary differential equations guarantees that the only solution of this equation on $\Sigma$ is $f \equiv 0$.  Therefore, $\Sigma$ is contained in the lightlike plane passing through the point $\bx_0$ and orthogonal to the (constant up to scalar multiple) null direction $\e_0$.
\end{proof}

Now suppose that $a_1 \neq 0$ at every point of $\Sigma$.  According to \eqref{a1-transformation}, there exists a 0-adapted frame on $\Sigma$ for which $a_1 \equiv 1$.  We will call such a frame {\em 1-adapted}.  Any two 1-adapted frames on $\Sigma$ must differ by a transformation of the form \eqref{0-adapted-frame-transformation} with $\mu=1$; i.e., 
\begin{equation}
\begin{bmatrix} \tilde{\e}_0 & \tilde{\e}_1 & \tilde{\e}_2 \end{bmatrix} = 
\begin{bmatrix} \e_0 & \e_1 & \e_2 \end{bmatrix}
\begin{bmatrix} 1 & \lambda & -\frac{1}{2}\lambda^2 \\[0.05in]
0 & 1 & -\lambda \\[0.05in]
0 & 0 & 1 \end{bmatrix}.  \label{1-adapted-frame-transformation}
\end{equation}
The Maurer-Cartan forms associated to two 1-adapted frames related by \eqref{1-adapted-frame-transformation} satisfy the equations (keeping in mind that $\w^2=0$):
\begin{align}
\tilde{\w}^0 & = \w^0 - \lambda \w^1, \notag \\
\tilde{\w}^1 & = \w^1 , \notag \\
\tilde{\w}^0_0 & = \w^0_0 - \lambda \w^1_0,  \label{1-adapted-MC-forms-transformation} \\
\tilde{\w}^1_0 & = \w^1_0, \notag \\
\tilde{\w}^0_1 & = \w^0_1 + \lambda \w^0_0 - \frac{1}{2} \lambda^2 \w^1_0 + d\lambda    . \notag
\end{align}

For any 1-adapted frame, we have $\w^1_0 = \w^1$.  Differentiating this condition yields
\begin{align*}
0 & = d\w^1_0 - d\w^1 \\
& = -\w^1_0 \& \w^0_0 + \w^1_0 \& \w^0 \\
& = (\w^0_0 - \w^0) \& \w^1_0 \\
& =  (\w^0_0 - \w^0) \& \w^1.
\end{align*}
By Cartan's lemma, we have
\[ \w^0_0 - \w^0 = a_2 \w^1 \]
for some function $a_2$ on $\Sigma$.

According to \eqref{1-adapted-MC-forms-transformation}, under a transformation of the form \eqref{1-adapted-frame-transformation} we have
\begin{align*}
\tilde{\w}^0_0 & = \w^0_0 - \lambda \w^1_0 \\
& = (\w^0 + a_2 \w^1) - \lambda \w^1  \\
& = (\tilde{\w}^0 + \lambda \w^1) +  (a_2 - \lambda) \w^1 \\
& = \tilde{\w}^0 + a_2 \tilde{\w}^1.
\end{align*}
Therefore, the transformed function $\tilde{a}_2$, defined by the equation
\[ \tilde{\omega}_0^0 = \tilde{\w}^0 + \tilde{a}_2 \tilde{\omega}^1, \]
given by
\begin{equation}
 \tilde{a}_2 = a_2. \label{a2-transformation}
\end{equation}
In other words, $a_2$ is an invariant function on $\Sigma$.

Since the value of $a_2$ is the same for all 1-adapted frames based at a point $\bx \in \Sigma$, we must consider additional derivatives in order to make further frame adaptations.  Our next step is to differentiate the equation $\w^0_0 = \w^0 + a_2 \w^1$:
\begin{align*}
0 & = d\w^0_0 - d(\w^0 + a_2 \w^1) \\
& = -\w^0_1 \& \w^1_0 + \w^0_0 \& \w^0 + \w^0_1 \& \w^1 - da_2 \& \w^1 + a_2 \w^1_0 \& \w^0 \\
& = -(da_2 \& \w^1 - a_2 \w^1_0 \& \w^0 - (\w^0 + a_2 \w^1) \& \w^0) \\
& = -(da_2 + 2 a_2 \w^0) \& \w^1.
\end{align*}
By Cartan's lemma, we have
\begin{equation}
 da_2 = -2 a_2 \w^0 + a_3 \w^1 \label{da2-equation}
\end{equation}
for some function $a_3$ on $\Sigma$.  According to \eqref{1-adapted-MC-forms-transformation}, under a transformation of the form \eqref{1-adapted-frame-transformation} we have
\begin{align*}
da_2 & = -2 a_2 \w^0 + a_3 \w^1 \\
& = -2 a_2 (\tilde{\w}^0 + \lambda \w^1) + a_3 \w^1 \\
& = -2 a_2 \tilde{\w}^0 + (a_3 - 2 a_2 \lambda) \tilde{\w}^1.
\end{align*}
Therefore, the transformed function $\tilde{a}_3$, defined by the equation
\[ da_2 = -2 a_2 \tilde{\w}^0 + \tilde{a}_3 \tilde{\w}^1, \]
given by
\begin{equation}
 \tilde{a}_3 = a_3 - 2 a_2 \lambda. \label{a3-transformation}
\end{equation}

From \eqref{a3-transformation}, we see that the ability to normalize $a_3$ depends on the value of $a_2$: if $a_2 \neq 0$, then there exists a 1-adapted frame with $a_3=0$; however, if $a_2=0$, then $a_3$ is an invariant.  So at this point we need to make another constant type assumption: we will assume that $a_2$ is either identically zero on $\Sigma$ or never equal to zero at any point of $\Sigma$.  

First consider the case where $a_2 \equiv 0$.

\begin{pro}\label{a2-zero-prop}
Let $\Sigma$ be a connected, regular lightlike surface in $\R^{2,1}$ with $a_1 \neq 0$, and suppose that $a_2 \equiv 0$ for every 1-adapted frame on $\Sigma$.  Then $\Sigma$ is contained in a lightlike cone in $\R^{2,1}$.
\end{pro}

\begin{proof}
If $a_2 \equiv 0$ on $\Sigma$, then we have $\w^0_0 = \w^0$, and so
\begin{align*}
 d\e_0 & = \e_0 \w^0_0 + \e_1 \w^1_0 \\
 & = \e_0 \w^0 + \e_1 \w^1 \\
 & = d\bx.
\end{align*}
Therefore,
\[ d(\bx - \e_0) = 0, \]
and so there exists a point $\mathbf{p} \in \R^{2,1}$ such that
\[ \bx - \e_0 = \mathbf{p} \]
for every point $\bx \in \Sigma$.  But then we have
\[ \bx - \mathbf{p} = \e_0, \]
from which it follows that $\bx - \mathbf{p}$ is a null vector for every $\bx \in \Sigma$.  Thus $\Sigma$ is contained in the lightlike cone based at $\mathbf{p}$, defined by the equation
\[ \langle \bx - \mathbf{p}, \bx - \mathbf{p} \rangle = 0. \]

\end{proof}

Now suppose that $a_2 \neq 0$ at every point of $\Sigma$; we will call such a surface {\em non-conical}.  According to \eqref{a3-transformation}, there exists a 1-adapted frame on $\Sigma$ for which $a_3 \equiv 0$.  We will call such a frame {\em 2-adapted}.  Any two 2-adapted frames on $\Sigma$ must differ by a transformation of the form \eqref{1-adapted-frame-transformation} with $\lambda=0$---which means that, in fact, there is a {\em unique} 2-adapted frame at each point of $\Sigma$. All that remains is to compute the remaining invariants associated to this frame.

So far, we know that the Maurer-Cartan forms associated to a 2-adapted frame satisfy the following relations:
\begin{align}
\w^2 & = 0, \notag \\
\w^1_0 & = \w^1,  \label{2-adapted-relations-1} \\
\w^0_0 & = \w^0 + a_2 \w^1. \notag
\end{align}
Moreover, from equation \eqref{da2-equation} and the fact that $a_3=0$, we have
\begin{equation}
 da_2 = -2 a_2 \w^0. \label{da2-equation-2}
\end{equation}
Differentiating equation \eqref{da2-equation-2} yields
\begin{align*}
0 & = -2 da_2 \& \w^0 + 2 a_2 (\w^0_0 \& \w^0 + \w^0_1 \& \w^1) \\
& = 2 a_2( \w^0_1 - a_2 \w^0 )\& \w^1.
\end{align*}
By Cartan's lemma, we have
\[ \w^0_1 - a_2 \w^0 = a_4 \w^1 \]
for some function $a_4$ on $\Sigma$.  Together with \eqref{2-adapted-relations-1}, this gives a complete set of relations among the Maurer-Cartan forms for a 2-adapted frame:
\begin{align}
\w^2 & = 0, \notag \\
\w^1_0 & = \w^1,  \label{2-adapted-relations-2} \\
\w^0_0 & = \w^0 + a_2 \w^1, \notag \\
\w^0_1 & = a_2 \w^0 + a_4 \w^1. \notag
\end{align}
According to the general theory of moving frames, the functions $a_2, a_4$ (and their derivatives) form a complete set of local invariants for non-conical lightlike surfaces.  These functions are not entirely arbitrary: we already know that $a_2$ must satisfy the differential equation \eqref{da2-equation-2}, and differentiating the equation for $\w^0_1$ in \eqref{2-adapted-relations-2} shows that $a_4$ must satisfy the differential equation
\begin{equation}
da_4 = (a_2^2 - 2 a_4) \w^0 + a_5 \w^1 \label{da4-equation}
\end{equation}
for some function $a_5$ on $\Sigma$.

\section{Canonical coordinates and normal forms}\label{normal-forms-sec}

In this section we will construct canonical local coordinates in a neighborhood of each point of a non-conical lightlike surface.  Then, using these coordinates, we will derive a local normal form for parametrizations of non-conical lightlike surfaces.

First, observe that
\begin{align*}
	d \omega^0 &= - \omega_0^0 \wedge \omega^0 - \omega_1^0 \wedge \omega^1 \\
		&= - ( \omega^0 + a_2 \omega^1) \wedge \omega^0 - (a_2 \omega^0 + a_4 \omega^1) \wedge \omega^1 \\
		&= a_2 \omega^0 \wedge \omega^1 - a_2 \omega^0 \wedge \omega^1 \\
		&= 0.
\end{align*}
By Poincar\'e's lemma (see \cite{Ivey}), locally there exists a function $u$ on $\Sigma$, unique up to an additive constant, such that
\begin{equation*}
 \w^0 = du. \label{w0-in-coords}
\end{equation*}
Next, we have
\begin{align}
	d \omega^1 &= - \omega_0^1 \wedge \omega^0 \notag \\
		&= - \omega^1 \wedge \omega^0  \label{domega1-equation} \\
		&= \omega^0 \wedge \omega^1. \notag
\end{align}
Since $d\w^1 \equiv 0 \mod{\w^1}$, the Frobenius theorem (see \cite{Ivey}) implies that locally there exist functions $v, \sigma$ on $\Sigma$, with $\sigma \neq 0$, such that
\begin{equation*}
 \w^1 = \sigma\, dv. \label{w1-in-coords-1}
\end{equation*}
The independence condition
\[ \w^0 \& \w^1 \neq 0 \]
then implies that $(u,v)$ form a local coordinate system on $\Sigma$, and so $\sigma$ may be regarded as a function $\sigma(u,v)$.  Then the structure equation \eqref{domega1-equation} becomes:
\[ \sigma_u\, du \& dv = \sigma\, du \& dv. \]
Thus the function $\sigma$ satisfies the partial differential equation
\[ \sigma_u = \sigma, \]
and hence
\[ \sigma(u,v) = e^u \sigma_0(v)\, dv \]
for some nonvanishing function $\sigma_0(v)$ on $\Sigma$.  Moreover, since 
\[ d(\sigma_0(v)\, dv) = 0, \]
Poincar\'e's lemma implies that locally, there exists a function $\tilde{v}$ on $\Sigma$ such that
\[ \sigma_0(v)\, dv = d\tilde{v}, \]
and then
\[ \w^1 = e^u\, d\tilde{v}. \]
Replacing $v$ with $\tilde{v}$ (and dropping the tilde) produces local coordinates $(u,v)$ on $\Sigma$ such that
\begin{equation}
\w^0 = du, \qquad \w^1 = e^u\, dv.  \label{local-coords}
\end{equation}
These coordinates are unique up to transformations of the form
\begin{equation}
 u \to u + r, \qquad v \to e^{-r} v + s, \label{coord-trans}
\end{equation}
with $r, s \in \R$.

Now consider the functions $a_2, a_4$ on $\Sigma$ as functions of $(u,v)$.  Equation \eqref{da2-equation-2} is equivalent to the PDE system
\[ (a_2)_u = -2 a_2, \qquad (a_2)_v = 0 \]
for the function $a_2(u,v)$; thus we have
\begin{equation}
 a_2 = c e^{-2u} \label{a2-solution}
\end{equation}
for some constant $c \in \R$.  Since we have assumed that $a_2 \neq 0$, we must have $c \neq 0$, and by a coordinate transformation of the form \eqref{coord-trans}, we can arrange that $c = \pm 1$.  Moreover, by reversing the orientation of $\Sigma$ if necessary, we can assume that $c=1$.

Now equation \eqref{da4-equation} is equivalent to the PDE
\[ (a_4)_u = e^{-4u} - 2 a_4  \]
for the function $a_4(u,v)$; the general solution to this equation is
\begin{equation}
a_4 = -\tfrac{1}{2} e^{-4u} + f(v) e^{-2u}, \label{a4-solution}
\end{equation}
where $f(v)$ is an arbitrary function of $v$.

It is straightforward to check that with $a_2, a_4$ as above, all the Maurer-Cartan structure equations \eqref{structure-eqns} are satisfied by the forms \eqref{2-adapted-relations-2}.  According to Cartan's  general theory of moving frames (see \cite{Griffiths74} for details), this guarantees the local existence of a lightlike surface $\Sigma$, together with a 2-adapted frame $(\e_0, \e_1, \e_2)$ along $\Sigma$ whose Maurer-Cartan forms are precisely the forms \eqref{2-adapted-relations-2}.  The equations \eqref{MC-forms-def} may be regarded as a compatible, overdetermined system of PDE's for the functions $\bx, \e_0, \e_1, \e_2$.  Given any solution of this system, the function $\bx(u,v)$ is a local parametrization of a non-conical lightlike surface in $\R^{2,1}$; moreover, our constructions up to this point show that locally, every non-conical lightlike surface arises in this way.

In order to examine the system \eqref{MC-forms-def} more closely, we first use equations \eqref{local-coords}, \eqref{a2-solution}, \eqref{a4-solution} to write all the Maurer-Cartan forms as linear combinations of $du, dv$:  \begin{align}
\w^2 & = 0, \notag \\
\w^0 & = du ,\notag \\
\w^1 & = e^u\, dv ,\label{MC-forms-in-coords}\\
\w^1_0 & = e^u\, dv ,\notag \\
\w^0_0 & = du + e^{-u}\, dv ,\notag \\
\w^0_1 & = e^{-2u}\, du + (-\tfrac{1}{2} e^{-3u} + f(v) e^{-u}) dv. \notag
\end{align}
Substituting these expressions into the equations for the $d\e_{\alpha}$ in \eqref{MC-forms-def} yields the following equations:
\begin{align}
d\e_0 & = \e_0\, du + (e^{-u} \e_0 + e^u \e_1)\, dv, \notag \\
d\e_1 & = e^{-2u} \e_0\, du + \left( (-\tfrac{1}{2} e^{-3u} + f(v) e^{-3u})\e_0 - e^u \e_2 \right)\, dv ,\label{de-equations-in-coords} \\
d\e_2 & = -\left( e^{-2u} \e_1 + \e_2 \right) \, du - \left( (-\tfrac{1}{2} e^{-3u} + f(v) e^{-3u})\e_1 + e^{-u} \e_2 \right) \, dv. \notag
\end{align}

The $du$-components of these equations are straightforward to integrate: from the first equation, we have
\[ (\e_0)_u = \e_0, \]
and therefore, 
\begin{equation}
 \e_0 = e^u G_0(v) \label{e0-solution-1}
\end{equation}
for some $\R^{2,1}$-valued function $G_0(v)$.  Substituting this expression into the $du$-component of the second equations yields
\[ (\e_1)_u = e^{-u} G_0(v), \]
and therefore,
\begin{equation}
 \e_1 = -e^{-u} G_0(v) + G_1(v) \label{e1-solution-1}
\end{equation}
for some $\R^{2,1}$-valued function $G_1(v)$.  Finally, the $du$-component of the third equation becomes
\[ (\e_2)_u = e^{-3u} G_0(v) - e^{-2u} G_1(v) - \e_2,   \]
and therefore,
\begin{equation}
 \e_2 = -\tfrac{1}{2}  e^{-3u} G_0(v) +  e^{-2u} G_1(v) + e^{-u} G_2(v) \label{e2-solution-1}
\end{equation}
for some $\R^{2,1}$-valued function $G_2(v)$.

Now, substituting \eqref{e0-solution-1}, \eqref{e1-solution-1}, \eqref{e2-solution-1} into the $dv$-components of equations \eqref{de-equations-in-coords} yields the following differential equations for the functions $G_0(v), G_1(v), G_2(v)$:
\begin{align}
G_0'(v) & = G_1(v), \notag \\
G_1'(v) & = f(v) G_0(v) - G_2(v), \label{G-odes} \\
G_2'(v) & = -f(v) G_1(v). \notag
\end{align}
This system is equivalent to the third-order ODE
\begin{equation}
G_0'''(v) - 2 f(v) G_0'(v) - f'(v) G_0(v) = 0 \label{G-3rd-order-ode}
\end{equation}
for the function $G_0(v)$, together with the equations
\begin{align*}
G_1(v) & = G_0'(v), \\
G_2(v) & = f(v) G_0(v) - G_0''(v) 
\end{align*}
for $G_1(v), G_2(v)$.

It turns out that solutions of \eqref{G-3rd-order-ode} are related to solutions of a second-order Sturm-Liouville equation associated to the function $f(v)$ (see, e.g., \cite{BC96} for a discussion of Sturm-Liouville operators):

\begin{pro}\label{Sturm-Liouville-prop}
Let $h(v)$ be any real-valued solution of the Sturm-Liouville equation
\begin{equation}
 \left( \left(\frac{d}{dt}\right)^2 -  \tfrac{1}{2} f(v) \right) h(v) = 0. \label{SL-equation}
\end{equation}
Then the function $g(v) = \left( h(v) \right)^2$ is a real-valued solution of \eqref{G-3rd-order-ode}.  Moreover, any real-valued solution $g(v)$ of \eqref{G-3rd-order-ode} is a linear combination of solutions of this form.
\end{pro}

\begin{proof}
Suppose that $h(v)$ satisfies \eqref{SL-equation}, and let $g(v) = \left( h(v) \right)^2$.  Then
\begin{align*}
g_0'(v) & = 2 h(v) h'(v); \\[0.1in]
g_0''(v) & = 2 \left(h'(v)\right)^2 + 2 h(v) h''(v) \\
& = 2 \left(h'(v)\right)^2 + f(v) \left(h(v)\right)^2; \\[0.1in]
g_0'''(v) & = 4 h'(v) h''(v) + f'(v) \left(h(v)\right)^2 + 2 f(v) h(v) h'(v) \\
& = 2 f(v) h(v) h'(v) + f'(v) \left(h(v)\right)^2 + 2 f(v) h(v) h'(v) \\
& = 2 f(v) g'(v) + f'(v) g(v);
\end{align*}
therefore $g(v)$ satisfies \eqref{G-3rd-order-ode}.

For the second statement, note that since \eqref{SL-equation} is a second-order, homogeneous, linear ODE, it has a 2-dimensional space of real-valued solutions spanned by two independent functions $h_1(v), h_2(v)$.  Since the solution space of \eqref{G-3rd-order-ode} must contain all linear combinations of squares of solutions of \eqref{SL-equation}, it must contain all linear combinations of the the three independent functions 
\begin{equation}
(h_1(v))^2, \qquad h_1(v) h_2(v), \qquad (h_2(v))^2. \label{g-solns-basis}
\end{equation}
But the solution space of \eqref{G-3rd-order-ode} is precisely 3-dimensional; hence every real-valued solution of \eqref{G-3rd-order-ode} is a linear combination of the functions \eqref{g-solns-basis}.

\end{proof}

It follows from Proposition \ref{Sturm-Liouville-prop} that the entries of any $\R^{2,1}$-valued solution $G_0(v)$ of \eqref{G-3rd-order-ode} must be linear combinations of the functions \eqref{g-solns-basis}.  Moreover, the inner product relations \eqref{adapted-frame-inner-products} must hold for the frame $(\e_0, \e_1, \e_2)$ given by \eqref{e0-solution-1}, \eqref{e1-solution-1}, \eqref{e2-solution-1}.  Fortunately, the symmetries of the Maurer-Cartan connection forms guarantee that if the relations \eqref{adapted-frame-inner-products} hold at any point of $\Sigma$, then they hold identically on $\Sigma$.  We can arrange this by imposing appropriate restrictions on the initial conditions for the ODE \eqref{G-3rd-order-ode}.

Finally, substituting \eqref{e0-solution-1}, \eqref{e1-solution-1} into the equation for $d\bx$ in \eqref{MC-forms-def} yields
\begin{align*}
 d\bx & = \e_0 \w^0 + \e_1 \w^1 \\
 & = \left( e^u G_0(v) \right)\, du + \left( -G_0(v) + e^u G_1(v) \right)\, dv \\
 & =   \left( e^u G_0(v) \right)\, du + \left( -G_0(v) + e^u G_0'(v) \right)\, dv.
\end{align*}
Integrating this equation (and translating so that $\bx(0,0) = \mathbf{0}$) gives
\begin{equation}
\bx(u,v) = e^u G_0(v) - \int_0^v G_0(\tau)\, d\tau - G_0(0).  \label{surface-param}
\end{equation}

By an action of the Minkowski isometry group, we can arrange that
\begin{equation*}
\bx(0,0) = \begin{bmatrix} 0 \\[0.05in] 0 \\[0.05in] 0 \end{bmatrix}, \qquad
\e_0 = \frac{1}{\sqrt{2}}\begin{bmatrix} 1 \\[0.05in] 1 \\[0.05in] 0 \end{bmatrix}, \qquad
\e_1 = \begin{bmatrix} 0 \\[0.05in] 0 \\[0.05in] 1 \end{bmatrix}, \qquad
\e_2 = \frac{1}{\sqrt{2}}\begin{bmatrix} -1 \\[0.05in] 1 \\[0.05in] 0 \end{bmatrix}.
\end{equation*}
This corresponds to choosing initial conditions
\begin{equation}
 G_0(0) = \frac{1}{\sqrt{2}}\begin{bmatrix} 1 \\[0.05in] 1 \\[0.05in] 0 \end{bmatrix}, \qquad 
G_0'(0) = \begin{bmatrix} \frac{1}{\sqrt{2}} \\[0.05in] \frac{1}{\sqrt{2}} \\[0.05in] 1 \end{bmatrix}, \qquad
G_0''(0) = \begin{bmatrix} \frac{1}{\sqrt{2}}(f(0) + \tfrac{3}{2}) \\[0.05in] \frac{1}{\sqrt{2}}(f(0) - \tfrac{1}{2}) \\[0.05in] 1 \end{bmatrix}
\label{G0-ICs}
\end{equation}
for the function $G_0(v)$.

Together with Propositions \ref{a1-zero-prop} and \ref{a2-zero-prop}, this proves the following classification theorem:

\begin{teo}\label{main-theorem}
Let $\Sigma \subset \R^{2,1}$ be a regular lightlike surface of constant type. Then up to a Minkowski isometry of the form \eqref{Minkowski-isometry}, $\Sigma$ is either:
\begin{itemize}
\item contained in a lightlike plane in $\R^{2,1}$;
\item contained in a lightlike cone in $\R^{2,1}$; or
\item locally parametrized by \eqref{surface-param}, where $G_0(v)$ is the unique solution of \eqref{G-3rd-order-ode} satisfying the initial conditions \eqref{G0-ICs}, and $f(v)$ is an arbitrary function of one variable.
\end{itemize}
\end{teo}

This theorem says that locally, a non-conical lightlike surface $\Sigma$ in $\R^{2,1}$ is completely determined by the choice of one arbitrary function $f(v)$, and $\Sigma$ can be constructed explicitly from solutions to the Sturm-Liouville equation \eqref{SL-equation} determined by $f$.

As a consequence of the formula \eqref{surface-param}, we obtain the following corollary:

\begin{cor}\label{ruled-cor}
Every lightlike surface $\Sigma$ of constant type in $\R^{2,1}$ is ruled by null lines.
\end{cor}

\begin{proof}
The statement is clearly true for lightlike planes and light cones, so suppose that $\Sigma$ is non-conical.
According to \eqref{surface-param}, each $u$-parameter curve with $v = v_0$ is a line parallel to the vector $G_0(v_0)$, which according to \eqref{e0-solution-1} is a multiple of the null vector $\e_0$ at each point of $\Sigma$.  Therefore, the $u$-parameter curves are null lines.
\end{proof}

\section{Examples}\label{examples-sec}

In this section we will construct examples of non-conical lightlike surfaces based on different choices for the function $f(v)$.  The simplest examples are those for which $f(v)$ is a constant function, since then the Sturm-Liouville equation \eqref{SL-equation} can be solved explicitly.

\begin{exa} $f(v)=0$.  In this case, solving the system \eqref{G-3rd-order-ode}, \eqref{G0-ICs} for $G_0(v)$ and substituting into \eqref{surface-param} produces the parametrization
\[ \bx(u,v) = \begin{bmatrix} 
\frac{1}{4 \sqrt{2}} \left(e^u (3 v^2 + 4v + 4) - (v^3 + 2 v^2 + 4 v + 4) \right) \\[0.1in]
\frac{1}{12 \sqrt{2}} \left( e^u (-3 v^2 + 12 v + 12) + (v^3 - 6 v^2 - 12 v - 12) \right) \\[0.1in]
\frac{1}{6} \left( e^u (3 v^2 + 6 v) - (v^3 + 3 v^2) \right)
\end{bmatrix}    \] 
for $\Sigma$.  This surface is shown in Figure \ref{f=0-fig}.  (The $x^0$-axis is shown as the vertical axis in all figures.)
\begin{figure}[h]
\includegraphics[width=3in]{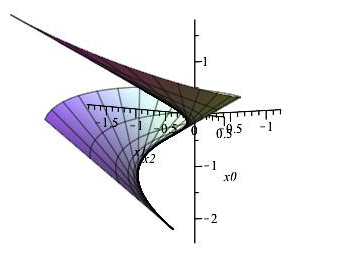}
\caption{Lightlike surface with $f(v)=0$}
\label{f=0-fig}
\end{figure}

\end{exa}

\begin{exa} $f(v)=1$.  In this case, solving the system \eqref{G-3rd-order-ode}, \eqref{G0-ICs} for $G_0(v)$ and substituting into \eqref{surface-param} produces the parametrization
\begin{multline*}
\bx(u,v) =  \\[0.1in]
\begin{bmatrix} 
\frac{1}{16} \left( e^{\sqrt{2}v}((5 \sqrt{2} + 4) e^u - 5 - 2\sqrt{2}) + e^{-\sqrt{2}v} ((5 \sqrt{2} - 4) e^u + 5 - 2\sqrt{2}) + 2\sqrt{2} (v - e^u - 2)
 \right)
 \\[0.1in]
\frac{1}{16} \left( e^{\sqrt{2}v}((\sqrt{2} + 4) e^u - 1 - 2\sqrt{2}) + e^{-\sqrt{2}v} ((\sqrt{2} - 4) e^u + 1 - 2\sqrt{2}) + 2\sqrt{2} (-3v + 3e^u - 2) \right)
 \\[0.1in]
\frac{1}{8} \left( e^{\sqrt{2}v}((2\sqrt{2} + 2) e^u - 2 - \sqrt{2}) + e^{-\sqrt{2}v} ((-2\sqrt{2} +2) e^u - 2 + \sqrt{2}) + 4 (v - e^u +1) \right)
\end{bmatrix}    
\end{multline*}
for $\Sigma$.  This surface is shown in Figure \ref{f=1-fig}. 
\begin{figure}[h]
\includegraphics[width=3in]{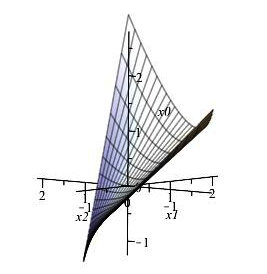}
\caption{Lightlike surface with $f(v)=1$}
\label{f=1-fig}
\end{figure}

\end{exa}

\begin{exa} $f(v)=-1$.  In this case, solving the system \eqref{G-3rd-order-ode}, \eqref{G0-ICs} for $G_0(v)$ and substituting into \eqref{surface-param} produces the parametrization
\[ \bx(u,v) = \begin{bmatrix} 
\frac{1}{8} \left( \sqrt{2}(2 - e^u) \cos(\sqrt{2}v) + (1 + 4 e^u) \sin(\sqrt{2}v) - \sqrt{2}(5v - 5 e^u + 6) \right)
 \\[0.1in]
\frac{1}{8} \left( \sqrt{2}(2 + 3e^u) \cos(\sqrt{2}v) + (-3 + 4 e^u) \sin(\sqrt{2}v) - \sqrt{2}(v - e^u + 6) \right)
 \\[0.1in]\frac{1}{4} \left( (2 - 2e^u) \cos(\sqrt{2}v) + \sqrt{2}(1 + 2 e^u) \sin(\sqrt{2}v) - (2v - 2e^u + 2) \right)
\end{bmatrix}    \] 
for $\Sigma$.  This surface is shown in Figure \ref{f=-1-fig}. 
\begin{figure}[h]
\includegraphics[width=3in]{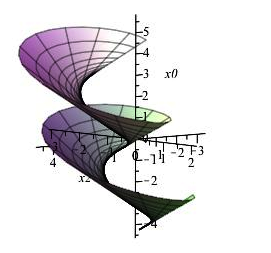}
\caption{Lightlike surface with $f(v)=-1$}
\label{f=-1-fig}
\end{figure}

\end{exa}

For nonconstant functions $f(v)$, equation \eqref{G-3rd-order-ode} generally cannot be solved explicitly for $G_0(v)$; however, we can still construct the corresponding surfaces $\bx(u,v)$ via numerical integration.  Examples with $f(v) = v$ and $f(v) = \sin(v)$ are shown in Figures \ref{f=v-fig} and \ref{f=sinv-fig}.
\begin{figure}[h]
\includegraphics[width=3in]{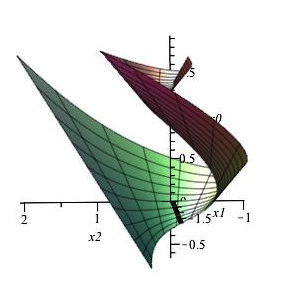}
\caption{Lightlike surface with $f(v)=v$}
\label{f=v-fig}
\end{figure}
\begin{figure}[h]
\includegraphics[width=3in]{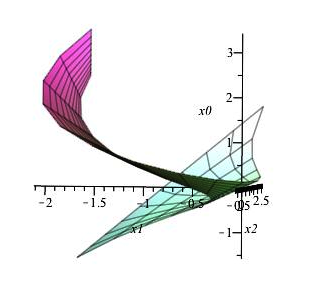}
\caption{Lightlike surface with $f(v)=\sin(v)$}
\label{f=sinv-fig}
\end{figure}

\clearpage

\bibliographystyle{amsplain}
\bibliography{lightlike_surfaces-bib}

\end{document}